\documentclass[12pt]{amsart}
\usepackage{xcolor}
\usepackage{ulem}
\usepackage{dirtytalk}
\usepackage{amsmath,amssymb,setspace,nicefrac, yhmath, amscd,eucal,dsfont}
\usepackage[active]{srcltx}
\usepackage[colorlinks, linkcolor=red, citecolor=blue, urlcolor=blue, hypertexnames=true]{hyperref}
\usepackage{amsrefs}
\usepackage{tikz-cd}
\allowdisplaybreaks
\usepackage[all]{xy}

\newcommand{\rk}{{\rm rk}}

\newcommand{\C}{{\mathbb C}}

\newcommand{\cC}{{\mathcal C}}

\newcommand{\cM}{{\mathcal M}}

\newcommand{\cN}{{\mathcal N}}

\def\d{\mathrm{d}}
\def\Prob{\mathrm{Prob}}
\def\id {\mathrm{id}}
\def\CA{\mathcal{A}}

\newtheorem{claim}{Claim}
\newtheorem{thm}{Theorem}[section]

\newtheorem{lemma}[thm]{Lemma}

\newtheorem{definition}[thm]{Definition}
\newtheorem{example}[thm]{Example}

\newtheorem{proposition}[thm]{Proposition}

\theoremstyle{definition}

\begin{document}
\title[Intermediate Subalgebras in crossed products of tensor products]{Non-commutative Factor theorem for tensor products of lattices in product groups}
\author[Amrutam]{Tattwamasi Amrutam}
\address{Institute of Mathematics of the Polish Academy of Sciences, ul. Sniadeckich 8, 00-656, Warszawa, Poland}
\email{tattwamasiamrutam@gmail.com}
\author[Jiang]{Yongle Jiang}
\address{School of Mathematical Sciences, Dalian University of Technology, Dalian, 116024, China}
\email{yonglejiang@dlut.edu.cn}
\author[Zhou]{Shuoxing Zhou}\address{\'Ecole Normale Sup\'erieure\\ D\'epartement de math\'ematiques et applications\\ 45 rue d'Ulm\\ 75230 Paris Cedex 05\\ FRANCE}
\email{shuoxing.zhou@ens.psl.eu}
\date{\today}
\begin{abstract}
We establish a non-commutative version of the Intermediate Factor Theorem for crossed products associated with product lattices. Given an irreducible lattice $\Gamma < G= G_1 \times \dots \times G_d$ in higher rank semisimple algebraic groups and a trace-preserving irreducible action $G \curvearrowright (\mathcal{N}, \tau)$, we show that every intermediate von Neumann algebra between $\mathcal{N}\rtimes\Gamma$ and $(L^\infty(G/P,\nu_P)\overline{\otimes}\mathcal{N})\rtimes\Gamma$ is again a crossed product of the form $(L^\infty(G/Q,\nu_Q)\overline{\otimes}\mathcal{N})\rtimes\Gamma$.
\end{abstract}
\maketitle
\section{Introduction}
The structure of intermediate von Neumann subalgebras in crossed product inclusions for actions of higher rank lattices
acts as a bridge between boundary theory,
ergodic rigidity, and operator algebras. Several works have been dedicated to the much more general framework of inclusions under assumptions on the action (and/or on the group), see for example, ~\cite{suzuki,amrutam2021intermediate,amrutam2023intermediate,amrutam2024crossed,amrutam2024non} and the references therein. The analytic framework of the crossed product
von Neumann algebras associated with lattice actions allow one to express the rigidity phenomena in a
non-commutative setting, giving rise to new \say{intermediate factor theorems}
within the operator-algebraic context (see, for example, \cite{houdayer2021noncommutative, boutonnet2023noncommutative}).

In the commutative case, intermediate factor theorems describe all the
intermediate $\Gamma$–equivariant  factors between the Poisson boundary times a measure preserving
$\Gamma$–action, that factors onto the measure preserving space.  Their non-commutative counterparts seek to characterize
intermediate von Neumann algebras in crossed products of the form $L(\Gamma) \subset
L^{\infty}(B,\nu_{B}) \rtimes \Gamma$,
where $(B,\nu_{B})$ is the $(\Gamma,\mu)$–Poisson boundary.
This viewpoint has proved fruitful in recent works of
Houdayer~\cite{houdayer2021noncommutative}, and
Boutonnet–Houdayer~\cite{boutonnet2023noncommutative},
who obtained complete descriptions of intermediate subfactors for a broad
class of crossed product inclusions.
In particular, their results show that, under appropriate freeness (read \say{singularity}, see for example, \cite{bader2022charmenability, bader2023charmenability}) and ergodicity assumptions, every intermediate von Neumann algebra arises again as a crossed product
by a suitable intermediate $\Gamma$–space.

For $1\leq i\leq d$, let $k_i$ be a local field. Let $\mathbf G_i$ be a simply
connected $k_i$-isotropic almost $k_i$-simple linear algebraic $k_i$-group such that $\sum_{i=1}^{d}\rk_{k_i}(\mathbf G_i) \geq 2$. Set $G_i = \mathbf G_i(k_i)$ and $G=\prod_{i=1}^d G_i$. Let $\mathbf P_i < \mathbf G_i$ be a minimal parabolic $k_i$-subgroup and set $P_i = \mathbf P_i(k_i)$ and $P=\prod_{i=1}^d P_i$. Our main result is as stated.
\begin{thm}
\label{thm:mainsplitting}
Let $G$ and $P$ be as above. Let $\Gamma<G$ be an irreducible lattice with finite center and set $\Lambda=\Gamma/Z(\Gamma)$. Let $(\cN,\tau)$ be a trace-preserving $G$-von Neumann algebra, on which each $G_i$ acts ergodically and $Z(\Gamma)$ acts trivially. Then every von Neumann algebra $\mathcal{M}$ with
\[
\mathcal{N} \rtimes\Lambda\subset \mathcal{M} \subset (L^\infty(G/P, \nu_P) \overline{\otimes} \mathcal{N}) \rtimes \Lambda
\]
is a crossed product of the form $(L^\infty(G/Q, \nu_Q) \overline{\otimes} \mathcal{N}) \rtimes \Lambda$ for some $P<Q<G$. 
\end{thm}
The above theorem provides new examples of a non-commutative analogue of the intermediate factor theorem
for product lattices, extending the rigidity result obtained in
Boutonnet-Houdayer~\cite{boutonnet2023noncommutative} to the tensor-product setting. 

The proof relies on four main ingredients and essentially follows the same strategy employed in \cite{boutonnet2023noncommutative}.
First, for $d\geq 2$, we make essential use of the notion of $G_i$–continuous elements,
introduced and studied in \cite{bader2022charmenability} and further dealt with in \cite{boutonnet2023noncommutative}. Roughly speaking, we show that $G_i$-continuous elements remain invariant under tensoring by the tracial von Neumann algebra $(\cN,\tau)$. 
\begin{thm}
\label{thm:invarince undertensor}
Following the notations as in Theorem \ref{thm:mainsplitting}, assume that $d\geq 2$. Then for each $1 \le i \le d$, the von Neumann subalgebra $\mathcal{M}_i$ of $G_i$–continuous elements in $\left(L^\infty(G/P, \nu_{P})\overline{\otimes} \mathcal{N}\right) \rtimes \Lambda$
is $L^\infty(G_i/P_i, \nu_{P_i})$.
\end{thm}
Second, for $d=1$, we need a version of the non-commutative Nevo-Zimmer theorem established in \cite{boutonnet2021stationary,bader2023charmenability} for a $k$-simple connected algebraic $k$-group with $\rk_k(\mathbf G) \geq 2$, which is a direct corollary of \cite[Theorem 5.4]{bader2023charmenability} (see Theorem~\ref{thm:mainembeddingprt}).

Thirdly, we pair the analytic techniques of
Suzuki~\cite{suzuki}
along with the structural arguments of Bader–Boutonnet–Houdayer–Peterson~\cite {bader2022charmenability} to achieve the desired crossed-product rigidity.  As mentioned earlier, we follow the strategy used by Boutonnet-Houdayer~\cite{boutonnet2023noncommutative}, and adapt it to our purpose of providing $\cM=\cM_0\rtimes \Lambda$. 

Finally, the classification of the intermediate von Neumann algebra $\mathcal{M}$ now reduces to understanding the structure of the $\Gamma$-invariant coefficient subalgebra $\mathcal{M}_0$. Using the non-commutative versions of Bader-Shalom \cite{amrutam2025non} and Stuck-Nevo-Zimmer (see Theorem~\ref{thm:nciftfield}), we show that every such intermediate algebra $\cM_0$ is a tensor product. Actually, we have the following more general result without any assumption on the center of lattice.
\begin{thm}
\label{thm:tensorproducts}    
Let $G$ and $P$ be as above. Let $\Gamma<G$ be an irreducible lattice. Let $(\cN,\tau)$ be a trace-preserving $G$-von Neumann algebra such that each $G_i$ acts ergodically on $\cN$. Then, any $\Gamma$-invariant intermediate von Neumann algebra 
$$\mathcal{N}\subset \mathcal{M}_0 \subset L^\infty(G/P, \nu_P) \overline{\otimes} \mathcal{N}$$
is actually a $G$-invariant subalgebra of the form $\mathcal{M}_0=L^\infty(G/Q, \nu_Q) \overline{\otimes} \mathcal{N}$ for some $P<Q<G$.
\end{thm}

\subsection*{Convenience} 
Throughout this paper, we assume that all von Neumann algebras have separable preduals.
\section{Proof of Main results}
\subsection{\texorpdfstring{$G_i$}{}-continuous elements}\label{subsection continuous element}
We shall use the following notation in this subsection. Let $d \ge 2$. For each $1\le i\le d$, let $G_i$ be a locally compact second countable group, $\mu_i \in \mathrm{Prob}(G_i)$ be an admissible Borel probability measure and $(B_i, \nu_{B_i})$ be the $(G_i,\mu_i)$–Poisson boundary.
Let
\[
(B, \nu_B) = \prod_{i=1}^d (B_i, \nu_{B_i}).
\]
Then $(B, \nu_B)$ is the $(G,\mu)$–Poisson boundary \cite[Corollary 3.2]{BS06}. Moreover, for each $1 \le i \le d$, let
\[
p_i : G \to G_i
\quad\text{and}\quad
\hat{p}_i : G \to \prod_{j \ne i} G_j
\]
be the canonical projections. Let $\Gamma<G$ be an \textbf{irreducible lattice}, i.e., $\hat{p}_i(\Gamma) <\prod_{j \ne i} G_j$ is a dense subgroup for each $i$. We recall the definition of $G_i$–continuous elements from \cite[Definition~3.3]{bader2022charmenability}.
\begin{definition} ($G_i$–continuous elements)
Let $\Gamma \curvearrowright \mathcal{M}$ be a $\Gamma$–von Neumann algebra with a normal faithful $\Gamma$-ucp map $\mathbb{E}:\mathcal{M}\to L^\infty(B,\nu_B)$. Let $\sigma : \Gamma \to \mathrm{Aut}(\mathcal{M})$ be the action by automorphisms. An element $x \in \mathcal{M}$ is said to be $G_i$–continuous if for every sequence $(\gamma_n) \subset \Gamma$ such that $p_i(\gamma_n) \to e$ in $G_i$, we have $\sigma_{\gamma_n}(x) \to x$ $\star$-strongly in $\mathcal{M}$.
\end{definition}
We remark that  the subset $\mathcal{M}_i$, that consists of all $G_i$–continous elements in $\mathcal{M}$
forms a $\Gamma$–invariant von Neumann subalgebra. Additionally, the action $\Gamma \curvearrowright \mathcal{M}_i$ extends to a continuous action of $G_i$ on $\mathcal{M}_i$ such that $G_j$ acts trivially on $\mathcal{M}_i$ for all $j \ne i$ (see  
\cite[Theorem~5.5]{bader2022charmenability}). One of the main tools used in the proof and main result is the invariance of $G_i$–continuous elements under the process of tensoring by a trace-preserving $\Gamma$–von Neumann algebra $(\mathcal{N},\tau)$. For convenience, we also denote by $\cC_i(\cM)$ the subalgebra of $G_i$-continuous elements in $\cM$. 

Let $(\mathcal{N},\tau)$ be a trace-preserving $\Gamma$-von Neumann algebra. Recall that for a lcsc group $H$, the \textbf{quasi-center} $QZ(H)$  is the (not necessarily closed) subgroup of all elements $h\in H$ for which the centralizer $Z_H(h)$ is open in $H$. We have $Z(H)<QZ(H)$.

\begin{thm}\label{thm:continuous elements are in 0 coefficient}
Assume that $\Gamma$ has finite center and $Z(\Gamma)$ acts trivially on $\cN$, and for each $1 \le i \le d$, $QZ(G_i)=Z(G_i)$. Let $\Lambda=\Gamma/Z(\Gamma)$. Then 
\[\mathcal{C}_i\left(\left(L^\infty(B, \nu_{B})\overline{\otimes} \mathcal{N}\right) \rtimes \Lambda\right)=\mathcal{C}_i\left(L^\infty(B_i, \nu_{B_i})\overline{\otimes}\mathcal{N}\right).\]
\end{thm}
\begin{proof}
Without any loss of generality, we may assume $d=2$ and consider the case that $i=1$. Following the same proof of \cite[Theorem 2.1 and Corollary 2.3]{boutonnet2023noncommutative}, we have
$$\cC_1\left(\left(L^\infty(B, \nu_{B})\overline{\otimes} \mathcal{N}\right) \rtimes \Lambda\right)= \mathcal{C}_1\left(L^\infty(B, \nu_{B})\overline{\otimes} \mathcal{N}\right).$$
Let $f\in L^\infty(B_2, L^\infty(B_1)\overline{\otimes}\cN)\cong L^\infty(B)\overline{\otimes}\cN$ be a $G_1$-continuous element and $y\in L^\infty(B_1)\overline{\otimes}\cN$ be an essential value of $f$. We only need to prove $\|f-1\otimes y \|_{\nu_{B_2}}=0$. Fix $\epsilon>0$. 
Let $E_\epsilon=\{b\in B_2\mid \|f(b)-y\|_{\nu_{B_1}\otimes\tau}<\epsilon\}$. Then by \cite[Lemma 2.2]{boutonnet2023noncommutative}, there exists a sequence $(\gamma_n)\subset\Gamma$ such that $\nu_{B_2}(p_2(\gamma_n)E_\epsilon)\to1$ and $p_1(\gamma_n)\to e$. 

To prove $\|f-1\otimes y \|_{\nu_{B_2}}=0$, we will follow the same proof of \cite[Lemma 5.4]{bader2022charmenability}. First, we will show that by passing to some subsequence of $(\gamma_n)$, we have $\|\sigma_{p_1(\gamma_n)}y- y\|_{\nu_{B_1}\otimes\tau}\to0$. %In the case of \cite[Lemma 5.4]{bader2022charmenability}, this is automatically guaranteed by $y\in L^\infty(B_1)$. However, here we only have $y\in L^\infty(B_1,\cN)$. 

Following the proof of \cite[Proposition 1.14]{NZ00}, by passing to a subsequence, we may assume $\frac{\d p_1(\gamma_n)^{-1}\nu_{B_1}}{\d\nu_{B_1}}\to \mathds{1}_{B_1}$ on $B_1$ pointwise almost everywhere. Since $\left(\frac{\d p_1(\gamma_n)^{-1}\nu_{B_1}}{\d\nu_{B_1}}\right)$ is uniformly $\|\cdot\|_{L^1(B_1,\nu_{B_1})}$-bounded, we have $\|\frac{\d p_1(\gamma_n)^{-1}\nu_{B_1}}{\d\nu_{B_1}}-\mathds{1}_{B_1}\|_{L^1(B_1,\nu_{B_1})}\to 0$. Hence we also have that the functional norm $\|(\nu_{B_1}\otimes\tau)\circ\sigma_{\gamma_n}-\nu_{B_1}\otimes\tau\|\to 0$. And we have
\begin{align*}
&\nu_{B_2}(E_\epsilon\cap p_2(\gamma_n)E_\epsilon) \cdot\|\sigma_{p_1(\gamma_n)}y- y\|_{\nu_{B_1}\otimes\tau}^2\\
\leq & 3\int_{E_\epsilon\cap p_2(\gamma_n)E_\epsilon}\|\gamma_nf(\gamma_n^{-1}b)-f(b)\|_{\nu_{B_1}\otimes\tau}^2\d \nu_{B_2}(b)\\
&+3\int_{E_\epsilon\cap p_2(\gamma_n)E_\epsilon}\|f(b)-y\|_{\nu_{B_1}\otimes\tau}^2\d \nu_{B_2}(b)\\
&+3\int_{E_\epsilon\cap p_2(\gamma_n)E_\epsilon}\|\gamma_n f(\gamma_n^{-1}b)-\gamma_ny\|_{\nu_{B_1}\otimes\tau}^2\d \nu_{B_2}(b)\\
:=& 3 X_n +3 Y_n +3 Z_n .
\end{align*}
When $n\to \infty$, since $\nu_{B_2}(E_\epsilon\cap p_2(\gamma_n)E_\epsilon)\to \nu_{B_2}(E_\epsilon)$, we have
$$ X_n \leq \|\sigma_{\gamma_n}f-f\|_{\nu_{B_2}\otimes\nu_{B_1}\otimes\tau}^2 \to 0;$$
$$ Y_n \leq \int_{E_\epsilon\cap p_2(\gamma_n)E_\epsilon}\epsilon^2\d \nu_{B_2}(b) \to \epsilon^2 \cdot\nu_{B_2}(E_\epsilon);$$
Let $f_n=(b\mapsto f(\gamma_n^{-1}b)-y)\in L^\infty(B_2, L^\infty(B_1)\overline{\otimes}\cN)$. Then $\|f_n\|_\infty\leq 2\|f\|_\infty$ and $\|f_n(b)\|_{\nu_{B_1}\otimes\tau}\leq \epsilon$ for $b\in p_2(\gamma_n)E_\epsilon$. Hence we have
\begin{align*}
 Z_n =& \int_{E_\epsilon\cap p_2(\gamma_n)E_\epsilon}\|\gamma_n(f_n(b))\|_{\nu_{B_1}\otimes\tau}^2 \d \nu_{B_2}(b)\\
 =&  \int_{E_\epsilon\cap p_2(\gamma_n)E_\epsilon}(\nu_{B_1}\otimes\tau)((\gamma_n-1)(f_n(b)^*f_n(b)))+ (\nu_{B_1}\otimes\tau)((f_n(b)^*f_n(b))) \d \nu_{B_2}(b)\\
\leq& \int_{E_\epsilon\cap p_2(\gamma_n)E_\epsilon}  \|(\nu_{B_1}\otimes\tau)\circ\sigma_{\gamma_n}-\nu_{B_1}\otimes\tau\| \cdot \|f_n\|_\infty^2 + \epsilon^2 \d \nu_{B_2}(b)\\
\to & \epsilon^2 \cdot \nu_{B_2}(E_\epsilon).
\end{align*}
Therefore,
$$\nu_{B_2}(E_\epsilon)\cdot\limsup_n \|\sigma_{p_1(\gamma_n)}y- y\|_{\nu_{B_1}\otimes\tau}^2 \leq \limsup_n 3( X_n + Y_n + Z_n )\leq 6\epsilon^2\cdot\nu_{B_2}(E_\epsilon).$$
Hence we get $\limsup_n \|\sigma_{p_1(\gamma_n)}y- y\|_{\nu_{B_1}\otimes\tau}^2\leq 6\epsilon^2$ for any $\epsilon>0$. Therefore, we have $\|\sigma_{p_1(\gamma_n)}y- y\|_{\nu_{B_1}\otimes\tau}\to 0$. 

Now by following the same proof of \cite[Lemma 5.4]{bader2022charmenability}, we have $\|f-1\otimes y \|_{\nu_{B_2}}=0$, which finishes the proof. 
\end{proof}
As mentioned earlier, the functoriality of $G_i$-continuous elements behaves well with respect to the tensor product by tracial von Neumann algebras. We make it precise below. 
\begin{thm}
\label{thm:continuouselements}
The subalgebra of $G_i$-continuous elements in $L^\infty(B_i,\nu_{B_i})\overline{\otimes}\mathcal{N}\cong L^\infty(B_i,\cN)$ is $L^\infty(B_i,\nu_{B_i})\overline{\otimes}\mathcal{C}_i(\mathcal{N})\cong L^\infty(B_i, \mathcal{C}_i(\mathcal{N}))$.
\end{thm}
\begin{proof}
Without any loss of generality, we may still assume $d=2$ and consider the case that $i=1$. The inclusion $L^\infty(B_1,\nu_{B_1})\overline{\otimes}\mathcal{C}_1(\mathcal{N})\subset\cC_1(L^\infty(B_1,\nu_{B_1})\overline{\otimes}\mathcal{N})$ is clear.

Conversely, assume that $f\in L^\infty(B_1,\cN)$ is a $G_1$-continuous element. Fix an essential value $y\in \cN$ of $f$. Then we only need to prove that $y\in\cC_1(\cN)$. For any $\epsilon>0$, let
$$E_\epsilon=\{b\in B_1\mid \Vert f(b)-y\Vert_{\tau,2}<\epsilon\}.$$
Then $\nu_{B_1}(E_\epsilon)>0$. Take a sequence $(\gamma_n)\subset\Gamma$ with $p_1(\gamma_n)\to e$ in $G_1$. Then for $b\in E_\epsilon\cap \gamma_n E_\epsilon$, we have
\begin{align*}
&\Vert \gamma_ny-y\Vert_{\tau,2}\\
\leq &\Vert\gamma_nf(\gamma_n^{-1}b)-f(b)\Vert_{\tau,2} +\Vert\gamma_nf(\gamma_n^{-1}b)-\gamma_n y\Vert_{\tau,2}+\Vert f(b)-y\Vert_{\tau,2}\\
\leq &\Vert\gamma_nf(\gamma_n^{-1}b)-f(b)\Vert_{\tau,2}+2\epsilon.  
\end{align*}
Hence
\begin{align*}
&\nu_{B_1}(E_\epsilon\cap \gamma_n E_\epsilon)\cdot \Vert \gamma_ny-y\Vert_{\tau,2}^2\\
\leq &8\epsilon^2\cdot\nu_{B_1}(E_\epsilon\cap \gamma_n E_\epsilon)+2\int_{E_\epsilon\cap \gamma_n E_\epsilon}\Vert\gamma_nf(\gamma_n^{-1}b)-f(b)\Vert_{\tau,2}^2\d \nu_{B_1}(b)\\
\leq & 8\epsilon^2\cdot\nu_{B_1}(E_\epsilon\cap \gamma_n E_\epsilon)+2 \Vert\sigma_{\gamma_n}(f)-f\Vert_{\nu_{B_1}\otimes\tau,2}^2.
\end{align*}
Since $\mathds{1}_{E_\omega}$ and $f$ are $G_1$-continuous, when $n\to \infty$, we have $\nu_{B_1}(E_\epsilon\cap \gamma_n E_\epsilon)\to \nu_{B_1}(E_\epsilon)$ and $\Vert\sigma_{\gamma_n}(f)-f\Vert_{\nu_{B_1}\otimes\tau,2}\to 0$. Therefore, we have $\limsup_n\Vert \gamma_ny-y\Vert_{\tau,2}^2\leq8\epsilon^2$ for any $\epsilon>0$. Hence $y\in\cC_1(\cN)$ and $f\in L^\infty(B_1,\cC_1(\cN))$.
\end{proof}
We now state the conditions under which we get that $\cC_i(\cN)=\C$. Following \cite{HM79, CCLRV11}, recall that a locally compact second countable group $H$ has the \textbf{Howe-Moore property} if any ergodic pmp action of $H$ on a diffuse probability space is necessarily mixing. For example, for a local field $k$, any simply connected $k$-isotropic almost $k$-simple linear algebraic $k$-group (even rank 1) has the Howe-Moore property \cite{HM79}.
\begin{thm}\label{Ni=C} Let $d\ge 2$. Assume that for each $1\le i\le d$, $G_i$ has the Howe-Moore property, and $(\cN,\tau)$ admits a trace-preserving $G$-action that extends the $\Gamma$-action. Also assume that $\cN$ is $G_i$-ergodic for each $1\le i\le d$. Then
\begin{itemize}
    \item [(i)]$\cC_i(\cN)=\C$;
    \item [(ii)] Any $\Gamma$-invariant intermediate von Neumann subalgebra $\cM_0$ with
    $$\cN\subset\cM_0\subset \cN\overline{\otimes}L^\infty(B,\nu_B)$$
    is of the form $\cM_0=\cN\overline{\otimes}L^\infty(C,\nu_C)$ for a $(G,\mu)$-boundary $(C,\nu_C)$.
\end{itemize}
\end{thm}
\begin{proof}
(i) Without any loss of generality, assume that $i=1$ and let $\hat{G}_1=\prod_{j\not=1}G_j$. By \cite[Theorem 5.6]{bader2022charmenability}, we have $\cC_1(\cN)\cong \mathrm{Ind}_\Gamma^G(\cN)^{\hat{G}_1}=(L^\infty(G/\Gamma)\overline{\otimes} \cN)^{\hat{G}_1}$. Since $N$ is already a $G$-von Neumann algebra, the induced $G$-action on $L^\infty(G/\Gamma)\overline{\otimes} \cN$ is just the diagonal $G$-action. Since $\Gamma<G$ is irreducible, following the proof of \cite[Lemma 5.1]{Pe14}, $G_j\curvearrowright (G/\Gamma,m_\Gamma)$ is ergodic for each $j$, where $m_\Gamma$ is the unique $G$-invariant probability measure on $G/\Gamma$. Moreover, since each $G_j$ has the Howe-Moore property, $G_j\curvearrowright (G/\Gamma,m_\Gamma)$ is mixing, hence weakly mixing, which is equivalent to being metrically ergodic in the case of pmp action. Hence by \cite[Lemma 2.4]{amrutam2025non}, $G_j\curvearrowright L^\infty(G/\Gamma)\overline{\otimes} \cN$ is ergodic for each $j$. And we have
$$\cC_1(\cN)\cong (L^\infty(G/\Gamma)\overline{\otimes} \cN)^{\hat{G}_1}\subset(L^\infty(G/\Gamma)\overline{\otimes} \cN)^{{G}_2}=\C.$$

(ii) By considering the induced $G$-actions, we have the following inclusions of $G$-von Neumann algebra
 $$L^\infty(G/\Gamma)\overline{\otimes} \cN\subset L^\infty(G/\Gamma)\overline{\otimes}\cM_0\subset (L^\infty(G/\Gamma)\overline{\otimes} \cN)\overline{\otimes}L^\infty(B,\nu_B).$$
 As in the discussion above, we know that $L^\infty(G/\Gamma,m_\Gamma)\overline{\otimes} (\cN,\tau)$ is a trace-preserving $G$-von Neumann algebra that is $G_i$-ergodic for each $i$. Hence by \cite[Theorem 1.1]{amrutam2025non}, there exists a $(G,\mu)$-boundary $(C,\nu_C)$ such that
 \begin{equation}\label{induced inclusion}
L^\infty(G/\Gamma)\overline{\otimes}\cM_0=(L^\infty(G/\Gamma)\overline{\otimes} \cN)\overline{\otimes}L^\infty(C,\nu_C).
 \end{equation}
 By considering the ucp map $m_\Gamma\otimes \id: L^\infty(G/\Gamma)\overline{\otimes}( \cN\overline{\otimes}L^\infty(B,\nu_B))\to \cN\overline{\otimes}L^\infty(B,\nu_B)$ acting on both sides of (\ref{induced inclusion}), we get
     $$\cM_0=\cN\overline{\otimes}L^\infty(C,\nu_C).$$
Note that above, when we write $m_{\Gamma}$, we identify the measure with the state obtained by integrating with respect to $m_{\Gamma}$.     
\end{proof}
\subsection{Intermediate Subalgebras as crossed products}
In this subsection, we let $\Lambda$ be a countable discrete group, $(B,\nu_B)$ be a nonsingular $\Lambda$-space, and $(\cN,\tau)$ be a trace-preserving $\Lambda$-von Neumann algebra. We denote by $\mathbb{E}$, the canonical conditional expectation onto $L^{\infty}(B,\nu_B)\overline{\otimes}\mathcal{N}$ from $\left(L^{\infty}(B,\nu_B)\overline{\otimes}\mathcal{N}\right)\rtimes\Lambda$.

We begin with the following proposition, which is essentially the same as in the proof of \cite[Theorem~3.6]{suzuki}. 
\begin{proposition}
\label{prop:essentiallyfree}
Consider a $\Lambda$-von Neumann subalgebra $\mathcal{M}$ of $(L^\infty(B, \nu_B) \overline{\otimes} \mathcal{N}) \rtimes \Lambda$. Assume that $\mathcal{M}$ contains $L^{\infty}(Y,\eta)$ where $(Y,\eta)$ is an essentially free factor of $(B,\nu_B)$. Then, $\mathbb{E}(\mathcal{M})\subset\mathcal{M}$.    
\end{proposition}
\begin{proof}
Let $a\in\mathcal{M}$ and $\epsilon>0$ be given. We can find $f_1,f_2,\ldots,f_k\in L^{\infty}(B,\nu_B)$, $n_1,n_2,\ldots,n_k\in\mathcal{N}$, and $s_1,s_2,\ldots,s_k\in\Lambda\setminus\{e\}$ such that $$a\approx_{\epsilon}\sum_{i=1}^k(f_i\otimes n_i)\lambda(s_i)+\mathbb{E}(a).$$
By our assumption, the action $\Lambda \curvearrowright (Y,\eta)$ is essentially free. From the proof of \cite[Theorem~3.6]{suzuki}, for each $a\in L^{\infty}(B,\nu_B)\rtimes_{\text{alg}}\Lambda$, we can find orthogonal projections $p_1,p_2,\ldots,p_n\in L^{\infty}(Y,\eta)$ such that $\sum_{i=1}^{n}p_i=1$ and $p_i(s_jp_i)=0$ for all $i=1,2,\ldots,n$, and $j=1,2,\ldots,k$. We now observe that
\begin{align*}
&\sum_{j=1}^n (p_j\otimes 1_{\mathcal{M}})\left(\sum_{i=1}^k(f_i\otimes n_i)\lambda(s_i)+\mathbb{E}(a)\right)(p_j\otimes 1_{\mathcal{M}})\\&=\sum_{j=1}^n\sum_{i=1}^k(p_jf_i\otimes n_i)(s_ip_j\otimes 1_{\mathcal{M}})\lambda(s)+\mathbb{E}(a)\\&= \sum_{j=1}^n\sum_{i=1}^k(p_j(s_ip_j)f_i\otimes n_i)\lambda(s)+\mathbb{E}(a)\\&=\mathbb{E}(a).  
\end{align*}
Now, a standard approximation argument using the triangle inequality yields that $\sum_{j=1}^np_jap_j$ approximates $\mathbb{E}(a)$. Since $L^{\infty}(Y,\eta)\subset\mathcal{M}$, it follows that $\sum_{j=1}^np_jap_j\in\mathcal{M}$. Consequently, $\mathbb{E}(a)\in\mathcal{M}$ for all $a\in\mathcal{M}$. The claim follows.  
\end{proof}
We isolate the following structural result about invariant subalgebras $\mathcal{M}$ of the ambient crossed product, which allows us to conclude the position of $\mathcal{M}$ when it is an invariant subalgebra. Let $\mathbb{E}_{\tau}:L^{\infty}(B,\nu_B)\overline{\otimes}\mathcal{N}\to L^{\infty}(B,\nu_B)$ denotes the slice map (which is also a $\Lambda$-equivariant conditional expectation) defined by sending $f\otimes n\to f\tau(n)$.
\begin{proposition}
\label{prop:insideN}
Let $\mathcal{M}$ be a $\Lambda$-invariant subalgebra of  $\left(L^{\infty}(B,\nu_B)\overline{\otimes}\mathcal{N}\right)\rtimes\Lambda$ that contains $\mathcal{N}$. Suppose that $\Phi|_{\mathcal{M}}:=\mathbb{E}_{\tau}\circ\mathbb{E}\big|_{\mathcal{M}}$ satisfies $\Phi(\cM)=\mathbb{C}$. Then, $\mathbb{E}(\cM)\subset\cN$.   
\end{proposition}
\begin{proof}
Fix $x\in \cM$, and assume that $\mathbb{E}(x)=\int_B^\oplus x_b\in L^\infty(B,\cN)\cong L^\infty(B)\overline{\otimes} \cN $. Take a weakly dense sequence $(y_n)\subset (\cN)_1$. Since $\cN\subset \cM$, we have $y_nx\in \cM$ for each $n$. Since $\Phi(y_nx)=\int_B^\oplus \tau(y_nx_b)\in \C$, there exists $c_n\in \C$ and a co-null subset $E_n\subset B$ such that $\tau(y_nx_b)=c_n$ for every $b\in E_n$. Let $E=\cap_n E_n\subset B$. Then $E$ is a co-null subset of $B$. Moreover, for any $b_1,b_2\in E$, for any $n\geq 1$, $\tau(y_nx_{b_1})=\tau(y_n x_{b_2})=c_n$. Since $(y_n)\subset (\cN)_1$ is weakly dense, we must have $x_{b_1}=x_{b_2}$. Therefore, there exists $x_0\in \cN$ such that $x_b=x_0$ for any $b\in E$. Moreover, since $E$ is a co-null subset of $B$. We have 
$$\mathbb{E}(x)=\int_B^\oplus x_b=1\otimes x_0\in \cN.$$
Therefore, by the arbitrariness of $x\in\cM$, we have $\mathbb{E}(\cM)\subset\cN$.
\end{proof}
\subsection{Tensor products and IFTs}
Let $k$ be a local field. Let $\mathbf G$ be a connected simply-connected semisimple algebraic $k$-group without $k$-anisotropic almost $k$-simple subgroups, and $\rk_k(\mathbf G) \geq 2$. Set $G = \mathbf G(k)$. Let $\mathbf P < \mathbf G$ be a minimal parabolic $k$-subgroup and set $P = \mathbf P(k)$. For a $G$-von Neumann algebra $M$, we denote by $\mathrm{Sub} (M)$ the space of von Neumann subalgebras of $M$, endowed with the \textbf{Effros-Mar\'{e}chal topology }\cite{Eff65,Mar73}. Recall that an \textbf{invariant random subalgebra (IRA)} is a $G$-invariant Borel probability measure $\nu\in\Prob(\mathrm{Sub} (M))$~\cite{amrutam2025amenable, zhou2024noncommutative}.
\begin{lemma}[Ergodic IRA on a product]
\label{lem:bundle}
Let \(\nu\) be an IRA on
\(\mathrm{Sub}\big(L^\infty(G/P,\nu_P)\big)\) that is $G$-irreducible, i.e., $\nu$ is $H$-ergodic for every non-central normal subgroup $H<G$.
Then there exists a \(G\)-invariant von Neumann subalgebra
\(L^\infty(G/Q,\nu_Q)\subset L^\infty(G/P,\nu_P)\) such that
\[
\nu=\delta_{L^\infty(G/Q,\nu_Q)}.
\]
\end{lemma}

\begin{proof} We denote by $E$ the collection $\mathrm{Sub}\big(L^\infty(G/P,\nu_P)\big)$.
Let us consider the following collection of section spaces.
\[
\mathcal U^\infty(E[\nu])\;:=\;\{\,s:E\to L^\infty(G/P,\nu_P):s \text{ is $\nu$-measurable}\,\},
\]
\[
S^\infty(E[\nu])\;:=\;\{\,s\in \mathcal U^\infty(E[\nu])\ :\ s(A)\in A
\ \text{for }\nu\text{-a.e. }A\in E\}.
\]
Then, we claim that
\[
L^\infty(E,\nu)\ \subseteq\ S^\infty(E[\nu])\ \subseteq\
L^\infty(E,\nu)\ \bar\otimes\ L^\infty(G/P,\nu_P).
\]
Indeed, the first inclusion holds because we can identity $f\in L^\infty(E,\nu)$ with $s_f: E\to L^\infty(G/P,\nu_P)$ defined by \[
s_f(A)=f(A)\cdot \mathbf 1_{L^{\infty}(G/P,\nu_P)},\] which lands in \(S^\infty(E[\nu])\).
For the second, identify $$\mathcal U^\infty(E[\nu])\cong L^\infty(E\times G/P,\nu\otimes \nu_P)
\cong L^\infty(E,\nu)\bar\otimes L^\infty(G/P,\nu_P).$$
    Now, using \cite[Theorem~1.1]{levit2017nevo}, there exists a $G$-invariant subalgebra $L^\infty(G/Q,\nu_Q)\subset L^{\infty}(G/P,\nu_P)$ such that
\[
S^\infty(E[\nu])\;=\;L^\infty(E,\nu)\ \bar\otimes\ L^\infty(G/Q,\nu_Q).
\]
Hence $\nu$-a.e. $A\in E$ is a subalgebra of $L^\infty(G/Q,\nu_Q)$. Also, for any $F\in L^\infty(G/Q,\nu_Q)$, the constant map $A\in E\mapsto F$ is contained in $S^\infty$, i.e., there exists a $\nu$-conull subset $E_F\subset E$ such that $F\in A$ for any $A\in E_F$. Take a weakly dense sequence $(F_n)\subset L^\infty(G/Q,\nu_Q)$ and let $E_0=\cap_n E_n$. Then $E_0$ is still a $\nu$-conull subset of $E$ and for any $A\in E_0$, $L^\infty(G/Q,\nu_Q)=\{ F_n\}''\subset A$. Therefore, $\nu=\delta_{L^\infty(G/Q,\nu_Q)}$. This completes the proof.
\end{proof}
With the structural understanding of $ G$-invariant tensor products established, we now turn our attention to $\Gamma$-invariant subalgebras within the crossed product. A central component of the rigidity argument is determining when an intermediate subalgebra contains a copy of a geometric boundary. This result acts as a non-commutative analogue of the Nevo-Stuck-Zimmer intermediate factor theorem. This generalizes \cite[Theorem~1.1]{levit2017nevo}.
\begin{thm}
\label{thm:nciftfield}
Let $(N, \tau)$ be a trace-preserving $G$-von Neumann algebra that is $G$-irreducible, i.e., $N$ is $H$-ergodic for every non-central normal subgroup $H<G$. Consider a $G$-invariant von Neumann subalgebra
\[
(N, \tau) \subseteq (M, \tau) \subseteq (N, \tau) \ \bar{\otimes} \ L^\infty(G / P, \nu_{P}),
\]
then
\[
(M, \tau) = (N, \tau) \ \bar{\otimes} \ L^\infty(G / Q, \nu_{Q})
\]
for some $G$-invariant subalgebra 
$L^\infty(G / Q, \nu_{Q}) \subseteq L^\infty(G / P, \nu_{P})$.
\end{thm}
\begin{proof}
We use the same techniques as above. The proof is divided into two steps. Let
\[
Z(N, \tau) = L^\infty(X, \mu).
\]
By our assumption, we see that $(X,\mu)$ is an irreducible $G$-space in the sense that every non-central normal subgroup acts ergodically on $(X,\mu)$. Write
\[
 M  = \int_X^\oplus  M _x \, \d\mu(x), 
\quad  N  = \int_X^\oplus  N _x \, \d\mu(x).
\]
Let 
\[
Y = \{x \in X \mid  N _x = \mathbb{C}\}, \quad Z = \{x \in X \mid  N _x \neq \mathbb{C}\}.
\]
By ergodicity, either $\mu(Y) = 0$ or $1$. Suppose that $\mu(Y) = 0$. Then for $\mu$-a.e.~$x \in X$, we have $N_x \neq \mathbb{C}$. In this case, we have by \cite[Theorem 4 (ii) in 3.3 of Chapter 3 in Part II]{Dixmier_book} that
\[
 N _x \subseteq  M _x \subseteq  N _x \ \bar{\otimes} \ L^\infty(G / P, \nu_{P}).
\]
Since $ N _x$ is a factor, using the Ge-Kadison splitting result~\cite{ge1996tensor}, we see that
\[
 M_x =  N_x \ \bar{\otimes} \ \CA_x,~\CA_x \subseteq L^\infty(G / P, \nu_{P}).
\]
Consider the \text{Subalg} map defined by
\[
\text{Subalg}:X \longrightarrow \text{Subalg}\left(L^\infty(G / P, \nu_{P})\right),
~x \longmapsto \CA_x.
\]
\begin{claim}
\label{claim:first}
$\text{Subalg}$ is $G$-equivariant and measurable.  
\end{claim}
\begin{proof}[Proof of Claim~\ref{claim:first}] 
Note that $G$-equivariance follows from the fact that both $M$ and $N$ are $G$-invariant and hence $gM_xg^{-1}=M_{gx}$ and $gN_xg^{-1}=N_{gx}$ hold for $\mu$-a.e. $x\in X$. 
\iffalse
% the following part is not needed since $nu$ is the pushforward of mu under the equivariant map Subalg.

For any Borel set $\mathcal J \subseteq \mathrm{SA}\left(L^\infty(G / P, \nu_{P})\right)$, we see that
\begin{align*}
\nu(\mathcal J) 
= \mu(\varphi^{-1}(\mathcal J))
&= \mu\big( \{ x\in X \mid  M _x \in \mathcal J \} \big).
\end{align*}
Thus
\begin{align*}
\nu(\sigma_g(\mathcal J)) 
&= \mu\big( \{ x\in X \mid  M _x \in \sigma_g(\mathcal J) \} \big) \\
&= \mu\big( \{ x\in X \mid \sigma_{g^{-1}}( M _x) \in \mathcal J \} \big) \\
&= \mu\big( \{ x\in X \mid  M _{g^{-1}\cdot x} \in \mathcal J \} \big) \\
&= \mu\big( g\{ x\in X \mid  M _x \in \mathcal J \} \big) \\
&= \nu(\mathcal J).
\end{align*}
Therefore, $\nu$ is $G$-invariant. 
\fi

Now, all we need to show is that the map $\text{Subalg}$ is measurable. Fix any subalgebra $\CA \subseteq L^\infty(G / P, \nu_{P})$, any finite subset $F \subset L^{\infty}(G/P,\nu_{P})$, and $\varepsilon>0$. Denote by $E_\CA: L^\infty(G / P, \nu_{P}) \to \CA$ the canonical trace-preserving conditional expectation.
Following \cite[Remark 2.11]{HW98}, we only need to show that
\[
\left\{ x\in X \mid \sup_{f\in F} \| E_\CA(f)-E_{\mathcal{A}_x}(f)\|_2^2 < \varepsilon \right\}
\text{ is $\mu$-measurable.}
\]
It suffices to check that for all $\xi,\eta\in L^2\left(G / P, \nu_{P}\right)$,  
the map
\[
x \mapsto \langle E_{ \mathcal{A}_x}(f)\xi,\eta\rangle
\]
is $\mu$-measurable. Note that
\[
E_{ M }\left(1\otimes f \right)
= \int_X^\oplus E_{ M _x}(f) \, \d \mu(x)=\int_X^{\oplus}1\otimes E_{\mathcal{A}_x}(f)d\mu(x).
\]
Here, $E_M: N\bar\otimes L^{\infty}(G/P,\nu_{P})\rightarrow M$, $E_{M_x}: N_x\bar{\otimes}L^{\infty}(G/P,\nu_{P})\rightarrow M_x$ and $E_{\mathcal{A}_x}: L^{\infty}(G/P,\nu_{P})\rightarrow \mathcal{A}_x$ are the conditonal expectations.
We see that \[
x \mapsto \langle E_{\mathcal{A}_x}(f)\xi,\eta\rangle
= \langle (E_ M (1\otimes f))_x (1\otimes\xi),1\otimes\eta\rangle,
\]
is $\mu$-measurable by indirect integral theory, say by \cite[Proposition 1 in 2.1 of Chapter 2 in Part II]{Dixmier_book}.
\end{proof}
Therefore
\[
\text{Subalg}_*\mu \in \text{IRA}(L^\infty(G / P), \nu_{P}).
\]
Using Lemma~\ref{lem:bundle}, we see that for some $G$-invariant von Neumann subalgebra
$L^\infty(G/Q, \nu_Q) \subseteq L^\infty(G / P, \nu_{P})$, 
\[\text{Subalg}_*\mu = \delta_{ L^\infty(G/Q, \nu_Q)}.
\]
Consequently, it follows that for $\mu$-a.e.~$x \in X$,
\[
\CA_x = L^\infty(G/Q, \nu_Q).
\]
Therefore, we have
\[
 M  = \int_X^\oplus  M _x \, \d\mu(x)
= \int_X^\oplus  N _x \ \bar{\otimes} \ L^\infty(G/Q, \nu_Q) \, \d\mu(x)
=  N  \ \bar{\otimes} \ L^\infty(G/Q, \nu_Q).
\]
Now, consider the case when $\mu(Y) = 1$. Then, in this case, we have
\[
 N  = \int_X^\oplus \mathbb{C} \, \d\mu(x) = L^\infty(X, \mu),
\]
and consequently, we are reduced to the case
\[
L^\infty(X, \mu) \subseteq  M  \subseteq L^\infty(X, \mu) \ \bar{\otimes} \ L^\infty(G / P, \nu_{P}),
\]
and $ M $ splits by \cite[Theorem 1.1]{levit2017nevo}.
\end{proof}
Let us now prove a version of the non-commutative Nevo-Zimmer theorem established in \cite{boutonnet2021stationary,bader2023charmenability} for a $k$-simple connected algebraic $k$-group with $\rk_k(\mathbf G) \geq 2$. The proof is immediate from \cite[Theorem 5.4]{bader2023charmenability}. We include the proof for the sake of completion. 
\begin{thm}
\label{thm:mainembeddingprt}
Let $k$ be a local field. Let $\mathbf G$ be a simply
connected $k$-isotropic almost $k$-simple linear algebraic $k$-group such that $\rk_k(\mathbf G) \geq 2$ and set $G = \mathbf G(k)$. Let $\mathbf P < \mathbf G$ be a minimal parabolic $k$-subgroup and set $P = \mathbf P(k)$. Let $\Gamma<G$ be a lattice and $(\cN,\tau)$ be a trace-preserving ergodic $G$-von Neumann algebra such that $Z(\Gamma)$ acts trivially on $\cN$. Let $\Lambda=\Gamma/Z(\Gamma)$ and $\Phi:=\mathbb{E}_{\tau}\circ\mathbb{E}: (L^\infty(G/P,\nu_P) \overline{\otimes} \mathcal{N}) \rtimes\Lambda \to L^\infty(G/P,\nu_P)$ be the canonical conditional expectation. Then for any $\Gamma$-invariant von Neumann algebra
$ \mathcal{M} \subset (L^\infty(G/P,\nu_P) \overline{\otimes} \mathcal{N}) \rtimes\Lambda
$, the following dichotomy holds:
\begin{itemize}
\item Either $\Phi$ is $\Gamma$-invariant, that is, $\Phi(\cM)=\mathbb{C}1$.
\item Or there exist a proper parabolic $k$-subgroup $\mathbf P < \mathbf Q < \mathbf G$ and a $\Gamma$-equivariant unital normal embedding $\iota : L^\infty(G/Q) \hookrightarrow \cM$ where $Q = \mathbf Q(k)$ such that $\mathbb{E} \circ \iota : L^\infty(G/Q) \hookrightarrow L^\infty(G/P)$ is the canonical unital normal embedding. In particular, $L^\infty(G/Q)\subset \cM\cap L^\infty(G/P)$.
\end{itemize}
\end{thm}
\begin{proof}
Since the Poisson boundary action $G\curvearrowright (G/P,\nu_P)$ is metrically ergodic and metrical ergodicity can be passed to the restricted actions of lattices \cite[Corollary 6.7]{BG17}, we have that $\Gamma$, and hence $\Lambda$, also act metrically ergodically on $(G/P,\nu_P)$. Following \cite[Theorem 4.1]{Mar91}, there exists a fully supported measure $\mu_\Gamma\in\Prob(\Gamma)$ such that $\nu_P$ is $\mu_\Gamma$-stationary. Let $\pi:\Gamma\to\Lambda$ be the quotient map and $\mu_\Lambda=\pi_*\mu_\Gamma$. Since $Z(\Gamma)$ acts trivially on $G/P$, we have that $\nu_P$ is also $\mu_\Lambda$-stationary. Following the same proof provided by R\'emi Boutonnet in \cite[Lemma 2.16]{kalantar2023invariant} and \cite[Lemma 2.4]{amrutam2025non}, since $\Lambda$ is ICC \cite[Lemma 2.6]{HI25} and $\nu_P\otimes \tau$ is $\mu_\Lambda$-stationary on $L^\infty(G/P,\nu_P) \overline{\otimes} \mathcal{N}$,  we have
$$((L^\infty(G/P,\nu_P) \overline{\otimes} \mathcal{N}) \rtimes\Lambda)^\Lambda=(L^\infty(G/P,\nu_P) \overline{\otimes} \mathcal{N})^\Lambda=\cN^\Lambda=\cN^\Gamma.$$
Following \cite[Subsection 2.1]{boutonnet2021stationary}, $\Gamma\curvearrowright \cN$ is ergodic iff the induced action $$G\curvearrowright \mathrm{Ind}_\Gamma^G(\cN)=L^\infty(G/\Gamma,m_\Gamma)\overline{\otimes}\cN$$ is ergodic. Note that $\cN$ is already a $G$-von Neumann algebra, hence the induced $G$-action on $L^\infty(G/\Gamma,m_\Gamma)\overline{\otimes}\cN$ is the diagonal action. Moreover, since $G$ has the Howe-Moore property and $G\curvearrowright (G/\Gamma,m_\Gamma)$ is ergodic, we have that $G\curvearrowright (G/\Gamma,m_\Gamma)$ is mixing and hence metrically ergodic. Moreover, since $\cN$ is $G$-ergodic, by \cite[Lemma 2.4]{amrutam2025non}, we have that $G\curvearrowright L^\infty(G/\Gamma,m_\Gamma)\overline{\otimes}\cN$ is ergodic, which is equivalent to $\cN^\Gamma=\C$. Hence $\Gamma\curvearrowright (L^\infty(G/P,\nu_P) \overline{\otimes} \mathcal{N}) \rtimes\Lambda$ is ergodic. In particular, $\Gamma\curvearrowright\cM$ is ergodic. By applying \cite[Theorem 5.4]{bader2023charmenability}, we finish the proof.
    \end{proof}
\subsection{Proof of Main results for algebraic groups}
\begin{proof}[Proof of Theorem~\ref{thm:invarince undertensor}]
Following \cite[the proof of Theorem 3.3 ]{boutonnet2023noncommutative}, we have $QZ(G_i)=Z(G_i)$ for $1\leq i \leq d$. Hence it follows from Theorem~\ref{thm:continuous elements are in 0 coefficient} that the von Neumann subalgebra $\mathcal{M}_i$ of $G_i$–continuous elements in $\left(L^\infty(G/P, \nu_{P})\overline{\otimes} \mathcal{N}\right) \rtimes\Lambda$ is exactly $\mathcal{C}_i\left(L^\infty(G_i/P_i, \nu_{P_i})\overline{\otimes} \mathcal{N}\right)$. It is a consequence of Theorem~ \ref{thm:continuouselements} that 
\[\mathcal{C}_i\left(L^\infty(G_i/P_i, \nu_{P_i})\overline{\otimes} \mathcal{N}\right)=L^\infty(G_i/P_i, \nu_{P_i})\overline{\otimes} \mathcal{C}_i(\mathcal{N}).\]
Finally, we obtain from Theorem~\ref{Ni=C} (i) that $\mathcal{C}_i(\mathcal{N})=\mathbb{C}$.
\end{proof}

\begin{proof}[Proof of Theorem~\ref{thm:tensorproducts}]
When $d\geq 2$, this follows form Theorem \ref{Ni=C} (ii).

When $d=1$, by considering the induced $G$-actions, we have the following inclusions of $G$-von Neumann algebra
 $$L^\infty(G/\Gamma)\overline{\otimes} \cN\subset L^\infty(G/\Gamma)\overline{\otimes}\cM_0\subset (L^\infty(G/\Gamma)\overline{\otimes} \cN)\overline{\otimes}L^\infty(G/P, \nu_P).$$
We can appeal to Theorem~\ref{thm:nciftfield} to conclude that 
 \[L^\infty(G/\Gamma)\overline{\otimes}\cM_0=L^\infty(G/\Gamma)\overline{\otimes} \cN\overline{\otimes}L^\infty(G/Q, \nu_Q), ~P<Q<G.\]
 By applying the ucp map $$m_\Gamma\otimes \id: L^\infty(G/\Gamma)\overline{\otimes}( \cN\overline{\otimes}L^\infty(G/P,\nu_P))\to \cN\overline{\otimes}L^\infty(G/P,\nu_P)$$ acting on both sides of the above equation, we get
$$\mathcal{M}_0=\cN\overline{\otimes}L^\infty(G/Q, \nu_Q).$$

\end{proof}

We now proceed to prove Theorem~\ref{thm:mainsplitting}. As is the case with the predecessors, the proof is divided into two cases. The first case is when the $\Gamma$-equivariant conditional expectation $\mathbb{E}_{\tau}\circ\mathbb{E}$ is $\Gamma$-invariant. In this situation, we show that $\mathcal{M}=\mathcal{N}\rtimes\Lambda$. The other case is when $\mathbb{E}_{\tau}\circ\mathbb{E}$ is not $\Gamma$-invariant. 
\begin{proof}[Proof of Theorem~\ref{thm:mainsplitting}] First, suppose that $\mathbb{E}_{\tau}\circ\mathbb{E}$ is $\Gamma$-invariant. It follows from Proposition~\ref{prop:insideN} that $\mathbb{E}(\cM)\subset\mathcal{N}$. Since $\mathcal{M}$ is an intermediate algebra, it follows that $\mathbb{E}(x\lambda(s)^*)\in\mathcal{N}$ for all $x\in\cM$ and for all $s\in\Lambda\setminus\{e\}$.
Therefore, it follows from \cite[Corollary~3.4]{suzuki} that 
\[
\cM = \cN\rtimes\Lambda.
\]
Suppose now that $\mathbb{E}_{\tau}\circ\mathbb{E}$ is not $\Gamma$-invariant. If $d=1$, then using Theorem~\ref{thm:mainembeddingprt}, we see that $L^{\infty}(G/Q,\nu_Q)\hookrightarrow\mathcal{M}$. If $d\ge 2$,  then, there exists $1\le i\le d$ such that $\mathcal{M}_i$, the collection of all \(G_i\)-continuous elements in $\mathcal{M}$ is non-trivial by \cite[Theorem 5.8]{bader2023charmenability}. Theorem \ref{thm:invarince undertensor} gives us that $L^{\infty}(G_i/P_i,\nu_{P_i})\hookrightarrow\mathcal{M}$. Either way, using Proposition~\ref{prop:essentiallyfree}, we see that $\mathbb{E}(\mathcal{M})\subset\mathcal{M}$. It now follows from \cite[Corollary~3.4]{suzuki} that $\mathcal{M}$ is a crossed product of the form $\mathcal{M}_0\rtimes\Lambda$, where $\cM_0$ is a $\Gamma$-von Neumann algebra with
\[
\mathcal{N} \subset \mathcal{M}_0 \subset L^\infty(G/P, \nu_P) \overline{\otimes} \mathcal{N}.
\] 
Now we can complete the proof with Theorem \ref{thm:tensorproducts} since each $G_i$ acts ergodically.
\end{proof}
We finish this section with a concrete example. 
\begin{example}
Let $\Gamma < G$ be as above, and consider the non-commutative Bernoulli shift. Let $(M_0, \tau_0)$ be a non-trivial tracial von Neumann algebra which is an amenable factor (say, for example, $M_2(\mathbb{C})$). Let
\[
\mathcal{N} = \bar{\otimes}_{g \in G} (M_0, \tau_0)_g
\]
be the infinite tensor product indexed by the group $G$, equipped with the translation action. Since $G$ is non-compact and the action is mixing, the restriction of this action to any non-compact factor $G_i$ is ergodic (indeed, mixing). In this setting, $\mathcal{N}$ is the hyperfinite II$_1$ factor. We can now apply Theorem~\ref{thm:mainsplitting} to conclude that all intermediate subalgebras between the crossed product $\mathcal{N} \rtimes \Gamma$ and the boundary extension $(L^\infty(G/P) \overline{\otimes} \mathcal{N}) \rtimes \Gamma$ is a crossed product of the form $(L^\infty(G/Q) \overline{\otimes} \mathcal{N}) \rtimes \Gamma$.
\end{example}

\section*{Acknowledgements}
This work was conducted while S.Z. was visiting the Institute of Mathematics of the Polish Academy of Sciences in Warsaw, partially supported by Ordinance No. 7/2023, which co-finances small scientific meetings from the statutory funds of IMPAN. Y.J. is partially supported by the National Natural Science Foundation of
China (Grant No. 12471118). We express our gratitude to Cyril Houdayer for taking the time to review an earlier draft and for several helpful suggestions. 
\bibliographystyle{amsalpha}
\bibliography{name}
\end{document}